\documentclass[a4paper,12pt]{amsart}

\usepackage{amsmath}
\numberwithin{equation}{section}
\usepackage[utf8]{inputenc}
\usepackage[T1]{fontenc}
\usepackage[margin=1.05in]{geometry}
\usepackage{amsmath,amssymb,amscd,amsthm,mathrsfs}
\usepackage{tikz-cd}
\usepackage{hyperref}
\usepackage{xcolor}
\usepackage{booktabs,tabularx,array}
\newtheorem{theorem}{Theorem}[section]
\newtheorem{proposition}[theorem]{Proposition}
\newtheorem{lemma}[theorem]{Lemma}
\newtheorem{corollary}[theorem]{Corollary}

\theoremstyle{definition}
\newtheorem{definition}[theorem]{Definition}

\newcommand{\OO}{\mathcal{O}}

\newcommand{\cF}{\mathcal{F}}

\newcommand{\End}{\operatorname{End}}

\newcommand{\Hom}{\operatorname{Hom}}
\newcommand{\PP}{\mathbb{P}}
\newcommand{\Res}{\operatorname{Res}}
\newcommand{\Spec}{\operatorname{Spec}}
\newcommand{\Sym}{\operatorname{Sym}}
\newcommand{\Tot}{\operatorname{Tot}}
\newcommand{\degc}{\operatorname{deg}}
\newcommand{\id}{\mathrm{id}}
\newcommand{\lengthop}{\operatorname{length}}
\newcommand{\rk}{\operatorname{rk}}
\newcommand{\sheafHom}{\mathcal{H}om}

\newcommand{\evres}{\operatorname{evres}}

\newcommand{\CC}{\mathbb C}
\newcommand{\longtwoheadrightarrow}{\mathrel{\relbar\joinrel\twoheadrightarrow}}
\newcommand{\longhookrightarrow}{\mathrel{\lhook\joinrel\longrightarrow}}

\title[Logarithmic Spectral Correspondence]{Logarithmic spectral correspondence for $V$--twisted Higgs bundles on punctured curves}

\author{Pradip Kumar}
\address{Department of Mathematics, Shiv Nadar University, NH91, Tehsil Dadri, Greater Noida, Uttar Pradesh 201314, India}
\email{Pradip.Kumar@snu.edu.in}


\begin{document}
\begin{abstract}
Let \(X\) be a smooth projective complex curve, \(P\subset X\) a reduced effective divisor, and \(X^{0}=X\setminus P\). We study logarithmic \(V\)-twisted Higgs bundles arising from a logarithmic Hecke compactification of a rank-two bundle on \(X^{0}\). We show that a pair of induced logarithmic line-twisted fields lifts uniquely exactly under explicit local Hecke conditions, and that the lift is integrable precisely when the fields commute. Fixing the compactified spectral curve \(Y\), we classify such Higgs bundles by pairs \((F,\,\vartheta)\), where \(F\) is a rank-one torsion-free sheaf on \(Y\) and \(\vartheta\) satisfies a marked spectral condition on a finite subscheme \(Z\subset Y\). This gives a logarithmic extension of the compact rank-two spectral correspondence of~\cite{ABK} to the punctured case. On the line-bundle locus, the moduli stack is canonically equivalent to
\(
\operatorname{Pic}^{d}(Y)\times A_Z .
\)
\end{abstract}

\maketitle
\tableofcontents
\section{Introduction}\label{sec:intro}

Let \(X\) be a smooth projective complex curve over \(\mathbb C\), let \(P\subset X\) be a reduced effective divisor, and set
\[
X_0:=X\setminus P.
\]
Fix a rank-two holomorphic vector bundle \(V_0\) on \(X_0\). A \(V_0\)-twisted Higgs field on a holomorphic vector bundle \(E_0\) over \(X_0\) is a section
\[
\theta_0\in H^0\bigl(X_0,\End(E_0)\otimes V_0\bigr)
\]
satisfying
\(
\theta_0\wedge \theta_0=0.
\)

The purpose of this paper is to develop a spectral correspondence for such twisted Higgs fields in the case where the twisting bundle is defined only on the punctured curve \(X_0\). Our approach is based on a logarithmic compactification of \(V_0\) across the punctures \(P\), followed by a rank-two Hecke description of the compactified twisting bundle. This produces a framework in which the punctured problem can be studied on the compact curve \(X\), while keeping track of the additional data contributed by the punctures.

This problem lies at the intersection of several existing directions. In \cite{GGN}, the authors study \(V\)-twisted Higgs bundles for a fixed holomorphic vector bundle \(V\) on a compact projective curve. In \cite{KSZ}, Kydonakis--Sun--Zhao establish a Beauville--Narasimhan--Ramanan type correspondence for twisted Higgs \(V\)-bundles and apply it to parabolic \(\mathrm{Sp}(2n,\mathbb R)\)-Higgs moduli; in that setting, the marked structure is built directly into the Higgs bundle, and the spectral side naturally leads to parabolic, orbifold, or root-stack geometry. On the other hand, in \cite{ABK}, the compact rank-two case is analyzed under a Hecke presentation of the twisting bundle by two line bundles, and the twisted Higgs field is reconstructed from a pair of line-twisted fields satisfying finitely many compatibility conditions.

The present work may be viewed as a logarithmic extension of the rank-two Hecke picture developed in \cite{ABK}. Starting from a rank-two bundle \(V_0\) on the punctured curve \(X_0\), we first choose a logarithmic compactification across \(P\). Once this compactification is fixed, the problem fits naturally into the general compact framework of \cite{GGN}, but with new puncture-dependent features that do not appear in the compact case. Our main point is that these puncture contributions can still be incorporated without enlarging the ambient spectral category: the spectral side remains in the classical compactified Beauville--Narasimhan--Ramanan framework of rank-one torsion-free sheaves on an ordinary compactified spectral curve. The extra logarithmic information is instead recorded by finitely many scheme-theoretic conditions on a finite marked subscheme lying over the punctures and the Hecke points. In this sense, the paper is complementary to \cite{KSZ}: rather than passing to parabolic or stack-theoretic spectral objects, we retain the usual compactified spectral curve and encode the puncture behaviour through explicit marked conditions.

A further feature of the logarithmic setting is that it produces genuinely new enhancement data on the spectral side. This has no direct analogue in \cite{ABK}. In particular, after establishing the logarithmic extension of the rank-two Hecke correspondence, we isolate an additional puncture-dependent parameter space determined by a finite marked spectral scheme. This leads to the affine scheme \(A_Z\), which governs the enhancement data on the line-bundle locus.

\medskip

Our starting point is the existence of logarithmic Hecke compactifications. By Proposition~\ref{prop:existence-log-hecke}, every rank-two holomorphic bundle \(V_0\) on \(X_0\) admits a logarithmic Hecke presentation: there exist line bundles \(S,L\) on \(X\), a reduced effective divisor \(D\subset X_0\), quotient maps
\[
q_x:(S\oplus L)_x\twoheadrightarrow \mathbb C_x \qquad (x\in D),
\]
and lines \(\ell_p\subset (S\oplus L)_p\) for \(p\in P\), giving an exact sequence
\begin{equation}\label{eq:intro-log-hecke-short}
0\longrightarrow V \longrightarrow W(P)\longrightarrow T_{\mathrm{tot}}\longrightarrow 0,
\qquad W:=S\oplus L,
\end{equation}
with \(V|_{X_0}\cong V_0\) and \(T_{\mathrm{tot}}\) supported on
\(
D_{\mathrm{tot}}:=D\cup P.
\)
Thus the compactified problem is determined by the logarithmic Hecke data
\(
(S,L,D,\{q_x\},\{\ell_p\}).
\)
Once such a presentation is fixed, every section
\(
\theta\in H^0(X,\End(E)\otimes V)
\)
induces a pair of logarithmic line-twisted fields
\[
\Theta\in H^0(X,\End(E)\otimes S(P)),
\qquad
\Theta'\in H^0(X,\End(E)\otimes L(P)).
\]

Our first main result shows that this reduction is exact.

\begin{quote}
\emph{Theorem \ref{thm:A}.}
A pair \((\Theta,\Theta')\) arises from a unique section
\[
\theta\in H^0(X,\End(E)\otimes V)
\]
if and only if explicit local logarithmic Hecke constraints hold at every point of \(D_{\mathrm{tot}}\). Under these conditions, the lift \(\theta\) is a \(V\)-twisted Higgs field if and only if \(\Theta\) and \(\Theta'\) commute.
\end{quote}

The spectral theory is built from the first field \(\Theta\). Its compactified spectral curve
\(
\varphi:\overline{Y}\longrightarrow X
\)
carries, via the compactified Beauville--Narasimhan--Ramanan correspondence, a rank-one torsion-free sheaf \(\mathcal F\) such that \(E\simeq \varphi_*\mathcal F\). When \(\Theta'\) commutes with \(\Theta\), it corresponds to an \(\mathcal O_{\overline{Y}}\)-linear morphism
\[
\vartheta:\mathcal F\longrightarrow \mathcal F\otimes \varphi^*L(P).
\]
The logarithmic Hecke constraints are then encoded on the finite marked spectral scheme
\(
Z:=\varphi^{-1}(D_{\mathrm{tot}})\subset \overline{Y}.
\)  This leads to the intrinsic form of the logarithmic spectral correspondence.

\begin{quote}
\emph{Theorem \ref{thm:B-marked}.}
Assume that \(\overline{Y}\) is integral. Then logarithmic \(V\)-twisted Higgs bundles whose associated logarithmic \(S\)-twisted field has compactified spectral curve \(\overline{Y}\) are naturally classified by pairs \((\mathcal F,\vartheta)\), where \(\mathcal F\) is a rank-one torsion-free sheaf on \(\overline{Y}\) and
\[
\vartheta:\mathcal F\longrightarrow \mathcal F\otimes \varphi^*L(P)
\]
satisfies an intrinsic marked spectral condition on \(Z\).
\end{quote}

Under additional hypotheses, this intrinsic condition becomes completely explicit.

\begin{quote}
\emph{Theorem \ref{thm:B}.}
If \(Z\) is reduced, if the slope maps determined by the logarithmic Hecke data are defined at all points of \(D_{\mathrm{tot}}\), and if \(\mathcal F\) is locally free along \(Z\), then the intrinsic marked spectral condition is equivalent to a finite system of scalar equations on the fibers over \(Z\).
\end{quote}

The final part of the paper studies the line-bundle locus of the logarithmic spectral correspondence. In this case the admissible enhancement data \((\mathcal F,\vartheta)\) are controlled by a fixed affine scheme depending only on the marked spectral data and not on the chosen line bundle. More precisely, in Definition~\ref{def:fixed-affine-scheme} we introduce the \emph{affine enhancement scheme}
\[
A_Z
:=
\Spec \CC
\times_{\underline{H^0(Z,\,\varphi^*L|_Z)}}
\underline{H^0\bigl(\overline{Y},\,\varphi^*L(P)\bigr)},
\]
which depends only on \((\overline{Y},Z,b_Z)\). If
\[
\mathfrak H^{\log,\mathrm{lb}}_{\overline{Y},d}
\]
denotes the category fibered in groupoids of logarithmic spectral data \((\mathcal F,\vartheta)\) with \(\mathcal F\in \mathbf{Pic}^d(\overline{Y})\), then we prove:

\begin{quote}
\emph{Theorem~\ref{thm:stack-product}.}
For each degree \(d\), there is a canonical equivalence of categories fibered in groupoids
\[
\mathfrak H^{\log,\mathrm{lb}}_{\overline{Y},d}
\;\simeq\;
\mathbf{Pic}^d(\overline{Y})\times A_Z.
\]
Moreover, by Corollary~\ref{cor:algebraicity-enhanced-stack}, this upgrades to a canonical equivalence of algebraic stacks.
\end{quote}

Therefore, on the line-bundle locus, the logarithmic enhancement problem separates cleanly into two independent choices: a line bundle on the compactified spectral curve and a point of the affine enhancement scheme \(A_Z\).

The paper is organized as follows. Section~\ref{sec:log-hecke} introduces the logarithmic Hecke compactification data and proves the existence of logarithmic presentations of rank-two bundles on the punctured curve. Section~\ref{sec:reconstruction} proves Theorem~\ref{thm:A}. Section~\ref{sec:spectral} develops the spectral correspondence, first in the intrinsic marked form of Theorem~\ref{thm:B-marked}, and then in the scalar form of Theorem~\ref{thm:B}. Section~\ref{sec:enhanced-hitchin} studies the line-bundle locus and proves the product decomposition of Theorem~\ref{thm:stack-product}.

Throughout the paper, all curves are defined over \(\mathbb C\), all divisors are reduced effective, and all sheaves are coherent unless stated otherwise.
\section{Logarithmic Hecke compactification data and induced fields}\label{sec:log-hecke}
Let \(X\) be a smooth projective complex curve, and let
\(
P\,=\,\sum_{j\,=\,1}^m p_j
\)
be a reduced effective divisor $(m\geq 1)$. Set 
\(
X_0\,:=\,X\setminus P\)  and \(j:X_0\longhookrightarrow X,
\)
and fix a holomorphic vector bundle \(V_0\) of rank \(2\) on \(X_0\).

For any holomorphic vector bundle \(M\) on \(X\), write
\(
M(P)\,:=\,M\otimes \OO_X(P).
\)
Since \(P\) is reduced, the quotient \(M(P)/M\) gives principal parts of meromorphic sections with at most simple poles along \(P\).  Equivalently, there is a natural short exact sequence
\begin{equation}\label{eq:principal-parts}
0\longrightarrow M\longrightarrow M(P)\xrightarrow{\ \operatorname{pp}_P^M\ }\bigoplus_{p\in P} M_p\longrightarrow 0,
\end{equation}
where \(M_p\,:=\,M\otimes_{\OO_X}\kappa(p)\) denotes the fiber of \(M\) at \(p\).  For each \(p\in P\), we write
\[
\beta_p^M:M(P)_p\longrightarrow M_p
\]
for the corresponding fiber map.  In a local coordinate \(z\) centered at \(p\), and for a local frame \(e_1,\dots,e_r\) of \(M\), the map \(\beta_p^M\) is given by
\[
\sum_{i\,=\,1}^r\left(\frac{a_{i,-1}}{z}+a_{i,0}+a_{i,1}z+\cdots\right)e_i
\longmapsto
\sum_{i\,=\,1}^r a_{i,-1}\,e_i(p).
\]

\subsection{Logarithmic Hecke presentations}
Fix line bundles \(S\) and \(L\) on \(X\), and set
\[
W\,:=\,S\oplus L.
\]
Let
\(
D\,=\,\sum_{i\,=\,1}^{\ell}x_i
\)
be a reduced effective divisor contained in \(X_0\), so that \(D\cap P\,=\,\varnothing\).  For each \(x\in D\), choose a nonzero one-dimensional quotient
\[
q_x:W_x\,=\,(S_x\oplus L_x)\longtwoheadrightarrow \mathbb C_x,
\]
and define
\(
T_D\,:=\,\bigoplus_{x\in D}\mathbb C_x.
\)
For each puncture \(p\in P\), choose a one-dimensional subspace
\(
\ell_p\subset W_p\,=\,(S\oplus L)_p,
\)
and set
\(
Q_p\,:=\,W_p/\ell_p.
\)
Applying \eqref{eq:principal-parts} to \(M\,=\,W\) gives a natural surjection
\[
\operatorname{pp}_P^W:W(P)\longtwoheadrightarrow \bigoplus_{p\in P}W_p.
\]

The data \(\{q_x\}_{x\in D}\) and \(\{\ell_p\}_{p\in P}\) determine a torsion sheaf
\(
T_{\mathrm{tot}}\,:=\,T_D\oplus \bigoplus_{p\in P}Q_p
\)
supported on
\(
D_{\mathrm{tot}}\,:=\,D\cup P,
\)
and a surjective morphism
\[
\xi:W(P)\longrightarrow T_{\mathrm{tot}}.
\]
At a point \(x\in D\), the stalk map \(\xi_x:W(P)_x\,\longrightarrow\, (T_{\mathrm{tot}})_x\) is the chosen quotient \(q_x\), using the canonical identification \(W(P)_x\,\simeq\, W_x\).  At a puncture \(p\in P\), the stalk map \(\xi_p\) is the composition
\[
W(P)_p\xrightarrow{\ \beta_p^W\ } W_p\longtwoheadrightarrow Q_p.
\]

\begin{definition}[Logarithmic Hecke presentation]\label{def:log-hecke}
A \emph{logarithmic Hecke presentation} of \(V_0\) is a choice of data \(S,L,D,\{q_x\}_{x\in D},\{\ell_p\}_{p\in P}\) as above together with an exact sequence
\begin{equation}\label{eq:log-hecke}
0\longrightarrow V \xrightarrow{\ \Phi\ } W(P)\xrightarrow{\ \xi\ } T_{\mathrm{tot}}\longrightarrow 0,
\end{equation}
such that \(V|_{X_0}\,\cong\, V_0\).
\end{definition}

\smallskip

Because \(X\) is a smooth curve and \(T_{\mathrm{tot}}\) is a torsion sheaf, the kernel \(V\,=\,\ker(\xi)\) is torsion-free, hence locally free.  Since \(W(P)\) has rank \(2\), so does \(V\).  Its degree is
\[
\degc(V)\,=\,\degc\bigl(W(P)\bigr)-\lengthop(T_{\mathrm{tot}})
\,=\,\degc(S)+\degc(L)+|P|-|D|.
\]

\medskip

\subsection{Existence and dependence on auxiliary data}
\label{sec:existence-canonicity}

We next address the basic existence question: given \(V_0\), can one always realize it by a logarithmic Hecke presentation?  The answer is yes.

\begin{proposition}\label{prop:existence-log-hecke}
Every rank-two holomorphic vector bundle \(V_0\) on \(X_0\) admits a logarithmic Hecke presentation.
\end{proposition}

\begin{proof}
Since \(X_0\) is an open Riemann surface, it is Stein.  By the Oka--Grauert principle, every rank-two holomorphic vector bundle on \(X_0\) is holomorphically trivial.  Choose a trivialization
\[
\tau:V_0\xrightarrow{\sim}\OO_{X_0}^{\oplus 2}.
\]

Set
\(
S\,=\,\,\OO_X\), and  \(L\,=\,\,\OO_X(-P),
\)
so that
\(
W(P)\,=\,(S\oplus L)(P)\,=\,\,\OO_X(P)\oplus \OO_X.
\)
Take \(D\,=\,\varnothing\).  For each \(p\in P\), define
\[
\ell_p\,=\,\{0\}\oplus L_p\subset S_p\oplus L_p\,=\,W_p.
\]
Then \(Q_p\,=\,W_p/\ell_p\,\simeq\, S_p\), and the induced map
\[
\xi:\OO_X(P)\oplus \OO_X\longrightarrow \bigoplus_{p\in P}Q_p
\]
is obtained by taking principal parts in the first summand and zero in the second summand.  Under the identification \(Q_p\,\simeq\, S_p\), this is exactly
\[
\xi\,=\,\bigoplus_{p\in P}\beta_p^S.
\]

We claim that \(\ker(\xi)\,=\,\,\OO_X\oplus \OO_X\).  Indeed, near a point \(p\in P\), a local section of \(\OO_X(P)\oplus \OO_X\) has the form
\[
\left(\frac{a_{-1}}{z}+a_0+a_1z+\cdots,\; b_0+b_1z+\cdots\right),
\]
and \(\xi\) gives only the coefficient \(a_{-1}\).  Thus the kernel condition is precisely that the first component be holomorphic at \(p\).  Since the second component is already holomorphic, the kernel is \(\OO_X\oplus \OO_X\).

Therefore \eqref{eq:log-hecke} is exact with \(V\,=\,\ker(\xi)\,\cong\, \OO_X^{\oplus 2}\).  Restricting to \(X_0\) and using \(\tau\), we obtain
\[
V|_{X_0}\,\cong\, \OO_{X_0}^{\oplus 2}\xrightarrow[\sim]{\tau^{-1}} V_0.
\]
Hence \(V_0\) admits a logarithmic Hecke presentation.
\end{proof}

Proposition~\ref{prop:existence-log-hecke} shows that existence is not restrictive in the holomorphic category.  The real issue is non-canonicity: the same bundle \(V_0\) may admit many different logarithmic Hecke presentations, depending on the auxiliary choices of \(S\), \(L\), \(D\), \(\{q_x\}\), and \(\{\ell_p\}\).

From this point on, we fix one logarithmic Hecke data: $(S,\,L,\,D,\,\{q_x\}_{x\in D},\,\{\ell_p\}_{p\in P})$ and presentation
\[
0\longrightarrow V \xrightarrow{\ \Phi\ } W(P)\xrightarrow{\ \xi\ } T_{\mathrm{tot}}\longrightarrow 0,
\qquad W\,=\,S\oplus L.
\]

\subsection{Slope maps and induced logarithmic fields}

Let
\[
p_1:W(P)\,\longrightarrow\, S(P),\qquad p_2:W(P)\,\longrightarrow\, L(P)
\]
be the natural projections.  Composing with \(\Phi\) gives bundle morphisms
\begin{equation}\label{eq:rho-hat}
\widehat{\rho}_1\,:=\,p_1\circ\Phi:V\longrightarrow S(P),\qquad
\widehat{\rho}_2\,:=\,p_2\circ\Phi:V\longrightarrow L(P).
\end{equation}

For each \(a\in D_{\mathrm{tot}}\), let
\[
\xi_a:W_a\,=\,S_a\oplus L_a\longrightarrow (T_{\mathrm{tot}})_a
\]
be the fiber map induced by \(\xi\).  Thus \(\xi_x\,=\,q_x\) for \(x\in D\), while \(\xi_p:W_p\longtwoheadrightarrow Q_p\) is the quotient map for \(p\in P\).  Since \((T_{\mathrm{tot}})_a\) is one-dimensional for every \(a\in D_{\mathrm{tot}}\), we can consider the component maps
\begin{equation}\label{eq:xi-fiber}
\xi_{1,a}:S_a\longrightarrow (T_{\mathrm{tot}})_a,\qquad
\xi_{2,a}:L_a\longrightarrow (T_{\mathrm{tot}})_a,
\end{equation}
obtained by restricting \(\xi_a\) to the two summands.

Assume that both maps in \eqref{eq:xi-fiber} are nonzero.  Since the source and target are then one-dimensional, they are isomorphisms, and there is a unique isomorphism
\begin{equation}\label{eq:rho-a}
\rho^a:S_a\longrightarrow L_a
\end{equation}
such that
\[
\xi_{2,a}\circ \rho^a\,=\,-\,\xi_{1,a}.
\]
Equivalently,
\[
\rho^a\,=\,-\,\xi_{2,a}^{-1}\circ \xi_{1,a},
\]
and \(\ker(\xi_a)\subset S_a\oplus L_a\) is the graph of \(\rho^a\).  We call \(\rho^a\) the \emph{slope map} at \(a\).

At a point \(x\in D\), this is the same fiberwise slope datum that appears in the ordinary Hecke presentation of \cite{ABK}.  At a puncture \(p\in P\), one has \(\ker(\xi_p)\,=\,\ell_p\), so the slope map \(\rho^p\) is defined exactly when \(\ell_p\) is transverse to both summands \(S_p\) and \(L_p\), or equivalently when
\[
\ell_p\cap S_p\,=\,\ell_p\cap L_p\,=\,0.
\]

Now let \(E\) be a holomorphic vector bundle on \(X\).  A \emph{\(V\)-twisted Higgs field} on \(E\) is a section
\[
\theta\in H^0\bigl(X,\,\, \End(E)\otimes V\bigr)
\]
satisfying
\begin{equation}\label{eq:integrability}
\theta\wedge\theta\,=\,0
\qquad\text{in}\qquad
H^0\bigl(X,\,\, \End(E)\otimes \wedge^2V\bigr).
\end{equation}

\begin{definition}[Logarithmic line--twisted field]
Let \(M\) be a line bundle on \(X\), and let \(E\) be a holomorphic vector bundle on \(X\).  A \emph{logarithmic \(M\)--twisted field} on \(E\) is a section of
\[
H^0\bigl(X,\,\, \End(E)\otimes M(P)\bigr).
\]
\end{definition}

Given \(\theta\in H^0(X,\,\, \End(E)\otimes V)\), composition with \eqref{eq:rho-hat} produces logarithmic line--twisted fields
\begin{equation}\label{eq:ThetaTheta}
\Theta\,:=\,(\id_{\End(E)}\otimes \widehat{\rho}_1)(\theta)\in H^0\bigl(X,\,\, \End(E)\otimes S(P)\bigr),
\end{equation}
\begin{equation}\label{eq:ThetaPrime}
\Theta'\,:=\,(\id_{\End(E)}\otimes \widehat{\rho}_2)(\theta)\in H^0\bigl(X,\,\, \End(E)\otimes L(P)\bigr).
\end{equation}
If \(E_0\,:=\,E|_{X_0}\), then the restrictions of \(\Theta\) and \(\Theta'\) to \(X_0\) are holomorphic sections of \(\End(E_0)\otimes S|_{X_0}\) and \(\End(E_0)\otimes L|_{X_0}\), respectively.

The integrability of \(\theta\) imposes a commutativity relation between the induced fields.

\begin{lemma}\label{lem:comm}
Let \(\theta\in H^0\bigl(X,\,\, \End(E)\otimes V\bigr)\) satisfy \eqref{eq:integrability}, and let \(\Theta,\,\Theta'\) be defined by \eqref{eq:ThetaTheta} and \eqref{eq:ThetaPrime}.  Then
\begin{equation}\label{eq:comm}
(\Theta'\otimes \id_{S(P)})\circ \Theta
\,=\,
(\Theta\otimes \id_{L(P)})\circ \Theta'
\end{equation}
as meromorphic homomorphisms
\[
E\longrightarrow E\otimes S(P)\otimes L(P).
\]
\end{lemma}
\begin{proof}
Set
\[
X^\circ\,:=\,X\setminus D_{\mathrm{tot}}.
\]
Since \(T_{\mathrm{tot}}\) is supported on \(D_{\mathrm{tot}}\), the sequence \eqref{eq:log-hecke} restricts over \(X^\circ\) to an isomorphism
\[
\Phi|_{X^\circ}:V|_{X^\circ}\xrightarrow{\sim} W(P)|_{X^\circ}\,\cong\, W|_{X^\circ}
\,\cong\, S|_{X^\circ}\oplus L|_{X^\circ}.
\]
With respect to this splitting, the section \(\theta|_{X^\circ}\) decomposes as
\[
\theta|_{X^\circ}\,=\,\Theta|_{X^\circ}+\Theta'|_{X^\circ},
\]
with \(\Theta|_{X^\circ}\in H^0(X^\circ,\, \End(E)\otimes S)\) and
\(\Theta'|_{X^\circ}\in H^0(X^\circ,\, \End(E)\otimes L)\).

Under the canonical identification \(\wedge^2(S\oplus L)\,\cong\, S\otimes L\), the relation \(\theta\wedge\theta\,=\,0\) becomes
\[
\Theta\wedge\Theta'+\Theta'\wedge\Theta\,=\,0
\qquad\text{on }X^\circ.
\]
Locally, choose frames \(s\) of \(S\) and \(\ell\) of \(L\), and write
\[
\Theta\,=\,A\otimes s,\qquad \Theta'\,=\,B\otimes \ell
\]
with \(A,B\in \End(E)\).  Then
\[
\theta\wedge\theta\,=\,(AB-BA)\otimes (s\wedge \ell),
\]
so \(\theta\wedge\theta\,=\,0\) is equivalent to \(AB\,=\,BA\).  This is precisely the identity
\[
(\Theta'\otimes \id_S)\circ \Theta
\,=\,
(\Theta\otimes \id_L)\circ \Theta'
\]
on \(X^\circ\). Both sides extend to meromorphic homomorphisms
\[
E\longrightarrow E\otimes S(P)\otimes L(P),
\]
and they agree on the dense open subset \(X^\circ\).  Hence they agree on all of \(X\).
\end{proof}
\section{Reconstruction from logarithmic line--twisted fields}\label{sec:reconstruction}
We continue with the fixed logarithmic Hecke presentation
\[
0\longrightarrow V \xrightarrow{\ \Phi\ } W(P)\xrightarrow{\ \xi\ } T_{\mathrm{tot}}\longrightarrow 0,
\qquad W\,=\,S\oplus L,
\]
and write \(D_{\mathrm{tot}}\,:=\,D\cup P\).  This section solves the reconstruction problem: given a pair of logarithmic line--twisted fields, determine when it comes from a section of \(\End(E)\otimes V\) and when the resulting lift is integrable.

For each \(a\in D_{\mathrm{tot}}\) and each line bundle \(M\in\{S,L\}\), define a linear map
\[
\beta_a^M:M(P)_a\longrightarrow M_a
\]
by the rule
\begin{itemize}
\item if \(a\in D\), then \(a\notin P\) and \(M(P)_a\,\simeq\, M_a\) canonically; in this case \(\beta_a^M\) is this identification;
\item if \(a\,=\,p\in P\), then \(\beta_p^M\) is the principal-part map from \eqref{eq:principal-parts}.
\end{itemize}
Thus \(\beta_a^M\) is an evaluation map at points of \(D\) and a residue map at points of \(P\).

Let \(E\) be a holomorphic vector bundle on \(X\), and let
\[
\Theta\in H^0\bigl(X,\,\, \End(E)\otimes S(P)\bigr),\qquad
\Theta'\in H^0\bigl(X,\,\, \End(E)\otimes L(P)\bigr)
\]
be logarithmic line--twisted fields.  For each \(a\in D_{\mathrm{tot}}\), we write
\[
\Theta_a:E_a\,\longrightarrow\, E_a\otimes S(P)_a,\qquad
\Theta'_a:E_a\,\longrightarrow\, E_a\otimes L(P)_a
\]
for their fiber values.

\begin{definition}[Local logarithmic Hecke constraint]\label{def:hecke-constraint}
For each \(a\in D_{\mathrm{tot}}\), define
\[
\mathsf{C}_{1,a}(\Theta)\,:=\,
(\id_{E_a}\otimes \xi_{1,a})\circ (\id_{E_a}\otimes \beta_a^S)\circ \Theta_a,
\]
\[
\mathsf{C}_{2,a}(\Theta')\,:=\,
(\id_{E_a}\otimes \xi_{2,a})\circ (\id_{E_a}\otimes \beta_a^L)\circ \Theta'_a,
\]
and
\[
\mathsf{C}_a(\Theta,\,\Theta')\,:=\,\mathsf{C}_{1,a}(\Theta)+\mathsf{C}_{2,a}(\Theta')
\in
\Hom\bigl(E_a,E_a\otimes (T_{\mathrm{tot}})_a\bigr).
\]
We say that \((\Theta,\,\Theta')\) satisfies the \emph{local logarithmic Hecke constraint at \(a\)} if
\begin{equation}\label{eq:hecke-constraint}
\mathsf{C}_a(\Theta,\,\Theta')\,=\,0.
\end{equation}
\end{definition}

At a puncture, this condition has a concrete residue interpretation.  If \(p\in P\), then \(\beta_p^S\) and \(\beta_p^L\) are residue maps, so we may identify
\[
(\id_{E_p}\otimes \beta_p^S)\Theta_p\,=\,\Res_p(\Theta),\qquad
(\id_{E_p}\otimes \beta_p^L)\Theta'_p\,=\,\Res_p(\Theta').
\]
Since \(\xi_p:W_p\longtwoheadrightarrow Q_p\,=\,W_p/\ell_p\) has kernel \(\ell_p\), the condition \eqref{eq:hecke-constraint} at \(p\) is equivalent to the statement that the residue pair
\[
\bigl(\Res_p(\Theta),\Res_p(\Theta')\bigr)\in \End(E_p)\otimes W_p
\]
lies in the distinguished line
\[
\End(E_p)\otimes \ell_p\subset \End(E_p)\otimes W_p.
\]

The next theorem is the basic reconstruction statement.

\begin{theorem}\label{thm:A}
Let \(E\) be a holomorphic vector bundle on \(X\), and let
\[
\Theta\in H^0\bigl(X,\,\, \End(E)\otimes S(P)\bigr),\qquad
\Theta'\in H^0\bigl(X,\,\, \End(E)\otimes L(P)\bigr)
\]
be a pair of logarithmic line--twisted fields.

\begin{enumerate}
\item[(a)] \textup{(lifting)} There exists a unique section
\[
\theta\in H^0\bigl(X,\,\, \End(E)\otimes V\bigr)
\]
inducing \((\Theta,\,\Theta')\) via \eqref{eq:ThetaTheta} and \eqref{eq:ThetaPrime} if and only if the local logarithmic Hecke constraints \eqref{eq:hecke-constraint} hold for every \(a\in D_{\mathrm{tot}}\).

\item[(b)] \textup{(integrability)} Under the equivalent conditions in \textup{(a)}, the lift \(\theta\) is a \(V\)-twisted Higgs field, that is, \(\theta\wedge\theta\,=\,0\), if and only if the commutativity relation \eqref{eq:comm} holds.
\end{enumerate}
\end{theorem}
\begin{proof}
For part \textup{(a)}, tensor the exact sequence \eqref{eq:log-hecke} with \(\End(E)\).  Since \(\End(E)\) is locally free, we obtain a short exact sequence of coherent sheaves
\[
0\longrightarrow \End(E)\otimes V
\longrightarrow
\End(E)\otimes W(P)
\xrightarrow{\ \Psi\ }
\End(E)\otimes T_{\mathrm{tot}}
\longrightarrow 0,
\]
where \(\Psi\,:=\,\id_{\End(E)}\otimes \xi\).  Under the decomposition
\[
\End(E)\otimes W(P)
\,\cong\,
\bigl(\End(E)\otimes S(P)\bigr)\oplus \bigl(\End(E)\otimes L(P)\bigr),
\]
set
\[
s\,:=\,\Theta\oplus \Theta'
\in H^0\bigl(X,\,\, \End(E)\otimes W(P)\bigr).
\]
By left exactness of global sections, a lift
\[
\theta\in H^0\bigl(X,\,\, \End(E)\otimes V\bigr)
\]
with image \(s\) exists if and only if \(s\in \ker H^0(\Psi)\), that is, if and only if
\[
\Psi(s)\,=\,0\in H^0\bigl(X,\,\, \End(E)\otimes T_{\mathrm{tot}}\bigr).
\]
Such a lift is unique because the map on global sections induced by \(\Phi\) is injective.

Since \(T_{\mathrm{tot}}\) is supported on the finite set \(D_{\mathrm{tot}}\), one has a canonical decomposition
\[
H^0\bigl(X,\,\, \End(E)\otimes T_{\mathrm{tot}}\bigr)
\,=\,
\bigoplus_{a\in D_{\mathrm{tot}}}
\Hom\bigl(E_a,E_a\otimes (T_{\mathrm{tot}})_a\bigr).
\]
Under this identification, the \(a\)-component of \(\Psi(s)\) is exactly
\[
\mathsf{C}_{1,a}(\Theta)+\mathsf{C}_{2,a}(\Theta')\,=\,\mathsf{C}_a(\Theta,\,\Theta').
\]
Therefore \(\Psi(s)\,=\,0\) if and only if \eqref{eq:hecke-constraint} holds for every \(a\in D_{\mathrm{tot}}\).  This proves \textup{(a)}.

For part \textup{(b)}, assume that the equivalent conditions in \textup{(a)} hold, and let \(\theta\) be the unique lift.  Set
\[
X^\circ\,:=\,X\setminus D_{\mathrm{tot}}.
\]
Since \(T_{\mathrm{tot}}|_{X^\circ}\,=\,0\), the sequence \eqref{eq:log-hecke} restricts over \(X^\circ\) to an isomorphism
\[
\Phi|_{X^\circ}:V|_{X^\circ}\xrightarrow{\sim} W(P)|_{X^\circ}\,\cong\, W|_{X^\circ}
\,\cong\, S|_{X^\circ}\oplus L|_{X^\circ}.
\]
Under this identification, \(\theta|_{X^\circ}\) corresponds to
\[
\Theta|_{X^\circ}\oplus \Theta'|_{X^\circ}.
\]
By the local computation used in the proof of Lemma~\ref{lem:comm}, the vanishing of \(\theta\wedge\theta\) on \(X^\circ\) is equivalent to the commutativity relation \eqref{eq:comm} on \(X^\circ\).

Now \(\theta\wedge\theta\) is a holomorphic section of
\[
\End(E)\otimes \wedge^2V,
\]
and the difference of the two sides of \eqref{eq:comm} is a holomorphic section of
\[
\Hom\bigl(E,E\otimes S(P)\otimes L(P)\bigr).
\]
A holomorphic section on the smooth irreducible curve \(X\) vanishes identically if and only if it vanishes on the dense open subset \(X^\circ\).  Hence \(\theta\wedge\theta\,=\,0\) on \(X\) if and only if \eqref{eq:comm} holds on \(X\).  This proves \textup{(b)}.
\end{proof}

The local constraint admits a more explicit form when the slope map is defined. Assume that for some \(a\in D_{\mathrm{tot}}\) both maps in \eqref{eq:xi-fiber} are nonzero, so that the slope map
\[
\rho^a:S_a\longrightarrow L_a
\]
is defined by \eqref{eq:rho-a}.  Since \(\xi_{2,a}\) is then an isomorphism and
\[
\rho^a\,=\,-\,\xi_{2,a}^{-1}\circ \xi_{1,a},
\]
the condition \eqref{eq:hecke-constraint} is equivalent to
\begin{equation}\label{eq:explicit-constraint}
(\id_{E_a}\otimes \beta_a^L)\circ \Theta'_a
\,=\,
(\id_{E_a}\otimes \rho^a)\circ (\id_{E_a}\otimes \beta_a^S)\circ \Theta_a
\end{equation}
in \(\Hom(E_a,\,E_a\otimes L_a)\).  This pointwise form is the version used in Section~\ref{sec:spectral}.

\section{Logarithmic compactified spectral correspondence}\label{sec:spectral}

Fix the logarithmic Hecke presentation
\[
0\longrightarrow V \xrightarrow{\ \Phi\ } W(P)\xrightarrow{\ \xi\ } T_{\mathrm{tot}}\longrightarrow 0,
\qquad W=S\oplus L,
\]
from Section~\ref{sec:log-hecke}, and let \(E\) be a holomorphic vector bundle on \(X\).
Let
\[
\Theta\in H^0\bigl(X,\, \End(E)\otimes S(P)\bigr),
\qquad
\Theta'\in H^0\bigl(X,\, \End(E)\otimes L(P)\bigr)
\]
be logarithmic line--twisted fields. 

\subsection{The first field and its compactified spectral curve}

Set
\[
X_0\,:=\,X\setminus P,
\qquad
\Theta_0\,:=\,\Theta|_{X_0}\in H^0\bigl(X_0,\, \End(E|_{X_0})\otimes S|_{X_0}\bigr).
\]
Let
\[
\pi_0:\Tot\bigl(S|_{X_0}\bigr)\longrightarrow X_0
\]
be the bundle projection, and let \(t_0\) be the tautological section of \(\pi_0^*S\).  The spectral curve of \(\Theta_0\) is the closed subscheme
\[
Y_0\subset \Tot\bigl(S|_{X_0}\bigr)
\]
defined by
\(\det\bigl(t_0-\pi_0^*\Theta_0\bigr)=0.\)

Next let
\(
\pi:\Tot\bigl(S(P)\bigr)\longrightarrow X
\)
be the total space of \(S(P)\), with tautological section
\(
t\in H^0\bigl(\Tot(S(P)),\,\pi^*S(P)\bigr).
\)
We define the compactified spectral curve of \(\Theta\) to be the closed subscheme
\(
Y\subset \Tot\bigl(S(P)\bigr)
\)
cut out by
\begin{equation}\label{eq:spectral-compact}
\det\bigl(t-\pi^*\Theta\bigr)=0.
\end{equation}
Since the defining equation is monic in the fiber coordinate, \(Y\longrightarrow X\) is finite of degree \(\rk(E)\), and
\[
Y\times_X X_0=Y_0.
\]

Let
\[
q:\PP_S\,:=\,\PP\bigl(\OO_X\oplus S(P)^{-1}\bigr)\longrightarrow X
\]
be the projective completion of \(\Tot(S(P))\), and let \(\Sigma_\infty\subset \PP_S\) be the infinity section.  Write
\[
\overline{Y}\subset \PP_S
\]
for the scheme-theoretic closure of \(Y_0\).  The homogenization of \eqref{eq:spectral-compact} is again monic in the affine fiber coordinate, hence does not vanish along \(\Sigma_\infty\).  Therefore
\[
\overline{Y}\cap \Sigma_\infty=\varnothing,
\]
so \(\overline{Y}\) is contained in the affine chart \(\Tot(S(P))\) and coincides there with \(Y\).  We write
\[
\varphi\,:=\,q|_{\overline{Y}}=\pi|_Y:\overline{Y}\longrightarrow X.
\]

Assume from now on that \(\overline{Y}\) is integral.  The ordinary compactified Beauville--Narasimhan--Ramanan correspondence for the line-bundle twist \(S(P)\) gives a bijection between
\begin{enumerate}
\item[(a)] pairs \((E,\,\Theta)\) whose compactified spectral curve is \(\overline{Y}\), and
\item[(b)] rank-one torsion-free sheaves \(\cF\) on \(\overline{Y}\),
\end{enumerate}
such that
\(
E\,\cong\, \varphi_*\cF,
\)
and under this identification \(\Theta\) is induced by multiplication by the tautological section \(t\).

\begin{definition}\label{def:compactified-sheaf}
Assume that \(\overline{Y}\) is integral.  A \emph{compactified spectral sheaf} for \((E,\,\Theta)\) is a rank-one torsion-free sheaf \(\cF\) on \(\overline{Y}\) such that \(E\,\cong\, \varphi_*\cF\) and \(\Theta\) is induced by multiplication by \(t\).
\end{definition}

The second field is encoded spectrally by \(\OO_{\overline{Y}}\)-linear endomorphisms of \(\cF\).

\begin{lemma}\label{lem:OYlinearity}
Assume that \(\overline{Y}\) is integral, and let \(\cF\) be the compactified spectral sheaf of \((E,\,\Theta)\).  Then the following are equivalent:
\begin{enumerate}
\item[(i)] the fields \(\Theta\) and \(\Theta'\) satisfy the commutativity relation \eqref{eq:comm};
\item[(ii)] under the identification \(E=\varphi_*\cF\), the field \(\Theta'\) is induced by an \(\OO_{\overline{Y}}\)-linear morphism
\[
\vartheta:\cF\longrightarrow \cF\otimes \varphi^*L(P).
\]
\end{enumerate}
\end{lemma}
\begin{proof}
Since \(\varphi\) is finite, adjunction and the projection formula give canonical identifications
\[
\Hom_{\OO_{\overline{Y}}}\!\bigl(\cF,\,\cF\otimes \varphi^*L(P)\bigr)
\,\cong\,
\Hom_{\OO_X}\!\bigl(\varphi_*\cF,\,\varphi_*\cF\otimes L(P)\bigr)
\,\cong\,
H^0\bigl(X,\, \End(E)\otimes L(P)\bigr).
\]
Under these identifications, an \(\OO_{\overline{Y}}\)-linear morphism \(\vartheta\) corresponds to an \(\OO_X\)-linear morphism
\[
\Theta':E\longrightarrow E\otimes L(P)
\]
which commutes with the \(\OO_{\overline{Y}}\)-action on \(\varphi_*\cF\).  Since \(\overline{Y}\subset \Tot(S(P))\) is generated over \(X\) by the tautological section \(t\), this is equivalent to commuting with multiplication by \(t\), i.e. with \(\Theta\).  This is precisely \eqref{eq:comm}.
\end{proof}

\subsection{The marked spectral scheme and the intrinsic condition}

The punctured problem is controlled not only by \(\overline{Y}\), but also by its finite subscheme above the marked divisor.

\begin{definition}\label{def:marked-spectral-scheme}
Let
\[
Z\,:=\,\overline{Y}\times_X D_{\mathrm{tot}},
\qquad
Z_D\,:=\,\overline{Y}\times_X D,
\qquad
Z_P\,:=\,\overline{Y}\times_X P.
\]
Then \(Z\) is a finite scheme over \(D_{\mathrm{tot}}\), and since \(D\cap P=\varnothing\),
\(
Z=Z_D\sqcup Z_P.
\)
We call \(Z\) the \emph{marked spectral scheme} attached to \(\Theta\).
\end{definition}

Let
\[
\psi:Z\longrightarrow D_{\mathrm{tot}}
\]
be the structure morphism.  Pulling back \(\xi\) along \(\psi\) gives a surjective morphism
\begin{equation}\label{eq:xiZ}
\xi_Z:\varphi^*W(P)|_Z=\psi^*\bigl(W(P)|_{D_{\mathrm{tot}}}\bigr)
\longrightarrow
\psi^*T_{\mathrm{tot}}=\varphi^*T_{\mathrm{tot}}|_Z .
\end{equation}

For \(M\in\{S,L,W\}\), let
\[
\beta_Z^M:\varphi^*M(P)|_Z\longrightarrow \varphi^*M|_Z
\]
be the morphism obtained componentwise from the maps \(\beta_a^M\) of Section~\ref{sec:reconstruction}: it is the identity over \(Z_D\) and the principal-part map over \(Z_P\).  By construction, \(\xi_Z\) factors as
\begin{equation}\label{eq:xiZ-factor}
\varphi^*W(P)|_Z
\xrightarrow{\ \beta_Z^W\ }
\varphi^*W|_Z
\xrightarrow{\ \overline{\xi}_Z\ }
\varphi^*T_{\mathrm{tot}}|_Z,
\end{equation}
where \(\overline{\xi}_Z\) is induced by the fiber maps
\[
\xi_a:S_a\oplus L_a\longrightarrow (T_{\mathrm{tot}})_a,
\qquad a\in D_{\mathrm{tot}}.
\]

Now assume that \(\overline{Y}\) is integral, let \(\cF\) be the compactified spectral sheaf of \((E,\,\Theta)\), and suppose that \(\Theta'\) satisfies \eqref{eq:comm}.  Let
\[
\vartheta:\cF\longrightarrow \cF\otimes \varphi^*L(P)
\]
be the corresponding \(\OO_{\overline{Y}}\)-linear morphism from Lemma~\ref{lem:OYlinearity}.  Together with multiplication by the tautological section,
\[
m_t:\cF\longrightarrow \cF\otimes \varphi^*S(P),
\]
it defines an \(\OO_Z\)-linear morphism
\begin{equation}\label{eq:GammaZ}
\Gamma_Z(\vartheta)\,:=\,(m_t,\,\vartheta)|_Z:
\cF|_Z\longrightarrow \cF|_Z\otimes \varphi^*W(P)|_Z.
\end{equation}

\begin{definition}\label{def:intrinsic-marked-admissibility}
We say that \(\vartheta\) satisfies the \emph{marked spectral condition} if
\begin{equation}\label{eq:intrinsic-marked-condition}
(\id_{\cF|_Z}\otimes \xi_Z)\circ \Gamma_Z(\vartheta)=0
\end{equation}
as a morphism
\[
\cF|_Z\longrightarrow \cF|_Z\otimes \varphi^*T_{\mathrm{tot}}|_Z.
\]
\end{definition}

This is the intrinsic spectral form of the local logarithmic Hecke constraints.

\begin{proposition}\label{prop:intrinsic-marked-condition}
Assume that \(\overline{Y}\) is integral.  Let \(\cF\) be the compactified spectral sheaf of \((E,\,\Theta)\), and let
\[
\vartheta:\cF\longrightarrow \cF\otimes \varphi^*L(P)
\]
be the \(\OO_{\overline{Y}}\)-linear morphism corresponding to \(\Theta'\) via Lemma~\ref{lem:OYlinearity}.  Then the following are equivalent:
\begin{enumerate}
\item[(i)] the pair \((\Theta,\,\Theta')\) satisfies the local logarithmic Hecke constraints \eqref{eq:hecke-constraint} for every \(a\in D_{\mathrm{tot}}\);
\item[(ii)] \(\vartheta\) satisfies the marked spectral condition \eqref{eq:intrinsic-marked-condition}.
\end{enumerate}
\end{proposition}
\begin{proof}
For each \(a\in D_{\mathrm{tot}}\), set
\[
Z_a\,:=\,\overline{Y}\times_X \Spec \kappa(a)=\varphi^{-1}(a)\subset Z.
\]
Since \(\varphi\) is finite, base change for coherent sheaves gives canonical identifications
\[
E_a=(\varphi_*\cF)\otimes \kappa(a)\,\cong\, H^0\bigl(Z_a,\,\cF|_{Z_a}\bigr),
\]
\[
E_a\otimes W(P)_a
\,\cong\,
H^0\bigl(Z_a,\,\cF|_{Z_a}\otimes \varphi^*W(P)|_{Z_a}\bigr),
\]
\[
E_a\otimes (T_{\mathrm{tot}})_a
\,\cong\,
H^0\bigl(Z_a,\,\cF|_{Z_a}\otimes \varphi^*T_{\mathrm{tot}}|_{Z_a}\bigr).
\]
Under these identifications, multiplication by \(t\) induces \(\Theta_a\), the morphism \(\vartheta\) induces \(\Theta'_a\), and
\[
\bigl((\id\otimes \xi_Z)\circ \Gamma_Z(\vartheta)\bigr)|_{Z_a}
\]
induces the map
\[
\mathsf{C}_a(\Theta,\,\Theta'):
E_a\longrightarrow E_a\otimes (T_{\mathrm{tot}})_a
\]
from Definition~\ref{def:hecke-constraint}.

Since \(Z_a\) is finite, it is affine.  A morphism of quasi-coherent sheaves on an affine scheme is zero if and only if the induced map on global sections is zero.  Therefore
\[
\mathsf{C}_a(\Theta,\,\Theta')=0
\iff
\bigl((\id\otimes \xi_Z)\circ \Gamma_Z(\vartheta)\bigr)|_{Z_a}=0.
\]
Finally, \(Z=\coprod_{a\in D_{\mathrm{tot}}} Z_a\), so vanishing on all \(Z_a\) is equivalent to \eqref{eq:intrinsic-marked-condition}.
\end{proof}

\subsection{Residual tautological data and the graph form}

Assume now that the slope map
\(
\rho^a:S_a\longrightarrow L_a
\)
is defined for every \(a\in D_{\mathrm{tot}}\); equivalently, \(\ker(\xi_a)\subset S_a\oplus L_a\) is the graph of \(\rho^a\) for every marked point.  Pulling back the collection \(\{\rho^a\}_{a\in D_{\mathrm{tot}}}\) along \(\psi:Z\longrightarrow D_{\mathrm{tot}}\) gives a morphism
\[
\rho_Z:\varphi^*S|_Z\longrightarrow \varphi^*L|_Z.
\]

\begin{definition}\label{def:lambdaZ}
The \emph{residual tautological section} is defined as 
\[
\lambda_Z\,:=\,\beta_Z^S\bigl(t|_Z\bigr)\in H^0\bigl(Z,\,\varphi^*S|_Z\bigr).
\]
We set
\(
b_Z\,:=\,\rho_Z(\lambda_Z)\in H^0\bigl(Z,\,\varphi^*L|_Z\bigr).
\)
\end{definition}

If
\(
\vartheta:\cF\longrightarrow \cF\otimes \varphi^*L(P)
\)
is \(\OO_{\overline{Y}}\)-linear, define
\begin{equation}\label{eq:varthetaZ}
\vartheta_Z\,:=\,(\id_{\cF|_Z}\otimes \beta_Z^L)\circ \vartheta|_Z:
\cF|_Z\longrightarrow \cF|_Z\otimes \varphi^*L|_Z.
\end{equation}
Since \(b_Z\) is a section of \(\varphi^*L|_Z\), multiplication by \(b_Z\) defines an \(\OO_Z\)-linear morphism
\[
m_{b_Z}:\cF|_Z\longrightarrow \cF|_Z\otimes \varphi^*L|_Z.
\]

\begin{proposition}\label{prop:marked-graph}
Assume that \(\overline{Y}\) is integral and that the slope maps \(\rho^a\) are defined for all \(a\in D_{\mathrm{tot}}\).  Then the marked spectral condition \eqref{eq:intrinsic-marked-condition} is equivalent to the equality
\begin{equation}\label{eq:marked-graph-condition}
\vartheta_Z=m_{b_Z}
\end{equation}
as morphisms
\(
\cF|_Z\longrightarrow \cF|_Z\otimes \varphi^*L|_Z.
\)
\end{proposition}
\begin{proof}
Using the factorization \eqref{eq:xiZ-factor}, condition \eqref{eq:intrinsic-marked-condition} is equivalent to
\[
(\id_{\cF|_Z}\otimes \overline{\xi}_Z)\circ
(\id_{\cF|_Z}\otimes \beta_Z^W)\circ \Gamma_Z(\vartheta)=0.
\]
By definition of \(\Gamma_Z(\vartheta)\), of \(\beta_Z^S\), and of \(\vartheta_Z\), this may be rewritten as
\[
(\id_{\cF|_Z}\otimes \overline{\xi}_Z)\circ
(m_{\lambda_Z},\,\vartheta_Z)=0.
\]
Restricting to \(Z_a\), the kernel of \(\overline{\xi}_Z|_{Z_a}\) is
\[
\cF|_{Z_a}\otimes \ker(\xi_a)
=
\cF|_{Z_a}\otimes \ell_a.
\]
Since \(\ell_a\subset S_a\oplus L_a\) is the graph of \(\rho^a\), the preceding vanishing condition on \(Z_a\) is equivalent to
\[
\vartheta_Z|_{Z_a}
=
m_{\rho^a(\lambda_Z|_{Z_a})}.
\]
By definition of \(b_Z\), this is exactly
\[
\vartheta_Z|_{Z_a}=m_{b_Z}|_{Z_a}.
\]
Since \(Z=\coprod_{a\in D_{\mathrm{tot}}} Z_a\), the claim follows.
\end{proof}

\subsection{The logarithmic spectral correspondence}

We can now state the spectral classification in its intrinsic form.

\begin{theorem}\label{thm:B-marked}
Assume that \(\overline{Y}\) is integral.  Then there is a natural bijection between the isomorphism classes of:
\begin{enumerate}
\item[(I)] logarithmic \(V\)-twisted Higgs bundles
\[
(E,\theta),
\qquad
\theta\in H^0\bigl(X,\, \End(E)\otimes V\bigr),
\qquad
\theta\wedge\theta=0,
\]
such that the associated logarithmic \(S\)-twisted field \(\Theta\) has compactified spectral curve \(\overline{Y}\);

\item[(II)] pairs \((\cF,\,\vartheta)\), where \(\cF\) is a rank-one torsion-free sheaf on \(\overline{Y}\) and
\[
\vartheta:\cF\longrightarrow \cF\otimes \varphi^*L(P)
\]
is an \(\OO_{\overline{Y}}\)-linear morphism satisfying the marked spectral condition \eqref{eq:intrinsic-marked-condition}.
\end{enumerate}
Under the correspondence,
\(
E=\varphi_*\cF,
\)
the field \(\Theta\) is induced by multiplication by \(t\), the field \(\Theta'\) is induced by \(\vartheta\), and \(\theta\) is the unique logarithmic \(V\)-twisted Higgs field lifting \((\Theta,\,\Theta')\).
\end{theorem}
\begin{proof}
Let \((E,\theta)\) be an object of \textup{(I)}, and let \((\Theta,\,\Theta')\) be the induced logarithmic line--twisted fields.  Since \(\overline{Y}\) is the compactified spectral curve of \(\Theta\), the compactified BNR correspondence gives a rank-one torsion-free sheaf \(\cF\) on \(\overline{Y}\) such that
\[
E\,\cong\, \varphi_*\cF,
\]
with \(\Theta\) induced by multiplication by \(t\).  Because \(\theta\wedge\theta=0\), Theorem~\ref{thm:A}\textup{(b)} implies \eqref{eq:comm}; hence Lemma~\ref{lem:OYlinearity} gives an \(\OO_{\overline{Y}}\)-linear morphism
\[
\vartheta:\cF\longrightarrow \cF\otimes \varphi^*L(P)
\]
inducing \(\Theta'\).  Finally, Theorem~\ref{thm:A}\textup{(a)} and Proposition~\ref{prop:intrinsic-marked-condition} show that \(\vartheta\) satisfies \eqref{eq:intrinsic-marked-condition}.  Thus \((\cF,\,\vartheta)\) is of type \textup{(II)}.

Conversely, let \((\cF,\,\vartheta)\) be an object of \textup{(II)}, and set
\[
E\,:=\,\varphi_*\cF.
\]
Multiplication by \(t\) defines a logarithmic \(S\)-twisted field
\[
\Theta:E\longrightarrow E\otimes S(P),
\]
and \(\vartheta\) defines a logarithmic \(L\)-twisted field
\[
\Theta':E\longrightarrow E\otimes L(P).
\]
By construction, \(\Theta'\) commutes with \(\Theta\), so \eqref{eq:comm} holds.  By Proposition~\ref{prop:intrinsic-marked-condition}, the pair \((\Theta,\,\Theta')\) satisfies the local logarithmic Hecke constraints.  Hence Theorem~\ref{thm:A} produces a unique section
\[
\theta\in H^0\bigl(X,\, \End(E)\otimes V\bigr)
\]
lifting \((\Theta,\,\Theta')\), and this lift is integrable.  This gives an object of \textup{(I)}.  The two constructions are inverse to one another.
\end{proof}

\subsection{Reduced fibers and the scalar form of the condition}

The intrinsic condition on \(Z\) is the natural one.  For the line-bundle locus studied in Section~\ref{sec:enhanced-hitchin}, we also need its pointwise form under the additional hypotheses used there.

For \(a\in D_{\mathrm{tot}}\) and a closed point \(y\in \varphi^{-1}(a)\), define
\[
\lambda_a(y)\,:=\,\beta_a^S\bigl(t(y)\bigr)\in S_a.
\]
If \(a=x\in D\), this is simply the value of the tautological section at \(y\).  If \(a=p\in P\), it is the residual value of the tautological coordinate at \(y\).

\begin{proposition}\label{prop:scalar}
Assume \(\overline{Y}\) is integral, \(Z=\varphi^{-1}(D_{\mathrm{tot}})\) is reduced, the slope maps \(\rho^a\) are defined for all \(a\in D_{\mathrm{tot}}\), and \(\cF\) is locally free at every point of \(Z\).  Let
\(
\vartheta:\cF\longrightarrow \cF\otimes \varphi^*L(P)
\)
be an \(\OO_{\overline{Y}}\)-linear morphism.  Then the following are equivalent:
\begin{enumerate}
\item[(i)] \(\vartheta\) satisfies the marked spectral condition \eqref{eq:intrinsic-marked-condition};
\item[(ii)] for every \(a\in D_{\mathrm{tot}}\) and every closed point \(y\in \varphi^{-1}(a)\), if
\(
\mu_y\in L(P)_a
\)
denotes the unique element corresponding to the induced map
\[
\vartheta(y):\cF(y)\longrightarrow \cF(y)\otimes L(P)_a
\]
under the canonical identification
\(
\Hom_\CC\bigl(\cF(y),\,\cF(y)\otimes L(P)_a\bigr)\,\cong\, L(P)_a,
\)
then
\begin{equation}\label{eq:scalar-condition}
\beta_a^L(\mu_y)=\rho^a\bigl(\lambda_a(y)\bigr)\in L_a.
\end{equation}
\end{enumerate}
\end{proposition}
\begin{proof}
Because \(Z\) is reduced and finite over \(X\), it is the disjoint union of its closed points.  Since \(\cF\) is locally free at every point of \(Z\), each \(\cF(y)\) is one-dimensional.  By Proposition~\ref{prop:marked-graph}, condition \textup{(i)} is equivalent to
\[
\vartheta_Z=m_{b_Z}.
\]
Restricting to the component \(y\in Z\), the left-hand side is multiplication by
\[
\beta_a^L(\mu_y)\in L_a,
\qquad a=\varphi(y),
\]
while the right-hand side is multiplication by
\[
b_Z(y)=\rho^a\bigl(\lambda_a(y)\bigr)\in L_a.
\]
Therefore \(\vartheta_Z=m_{b_Z}\) is equivalent to \eqref{eq:scalar-condition} for every closed point \(y\in Z\).
\end{proof}

We now record the pointwise version of the correspondence used in Section~\ref{sec:enhanced-hitchin}.

\begin{theorem}\label{thm:B}
Assume that \(\overline{Y}\) is integral, that \(Z=\varphi^{-1}(D_{\mathrm{tot}})\) is reduced, and that the slope maps \(\rho^a\) are defined for all \(a\in D_{\mathrm{tot}}\).  Then there is a natural bijection between the isomorphism classes of:
\begin{enumerate}
\item[(I)] logarithmic \(V\)-twisted Higgs bundles
\[
(E,\theta),
\qquad
\theta\in H^0\bigl(X,\, \End(E)\otimes V\bigr),
\qquad
\theta\wedge\theta=0,
\]
such that the associated logarithmic \(S\)-twisted field \(\Theta\) has compactified spectral curve \(\overline{Y}\), and the associated compactified spectral sheaf is locally free at every point of \(Z\);

\item[(II)] pairs \((\cF,\,\vartheta)\), where \(\cF\) is a rank-one torsion-free sheaf on \(\overline{Y}\), locally free at every point of \(Z\), and
\[
\vartheta:\cF\longrightarrow \cF\otimes \varphi^*L(P)
\]
is an \(\OO_{\overline{Y}}\)-linear morphism satisfying \eqref{eq:scalar-condition} for every \(a\in D_{\mathrm{tot}}\) and every closed point \(y\in \varphi^{-1}(a)\).
\end{enumerate}
Under the correspondence,
\(
E=\varphi_*\cF,
\)
the field \(\Theta\) is induced by multiplication by \(t\), the field \(\Theta'\) is induced by \(\vartheta\), and \(\theta\) is the unique logarithmic \(V\)-twisted Higgs field lifting \((\Theta,\,\Theta')\).
\end{theorem}
\begin{proof}
This is immediate from Theorem~\ref{thm:B-marked} and Proposition~\ref{prop:scalar}.
\end{proof}

Theorem~\ref{thm:B-marked} is the intrinsic spectral statement.  The extra hypotheses in Proposition~\ref{prop:scalar} and Theorem~\ref{thm:B} are needed only to pass from the scheme-theoretic condition on \(Z\) to the pointwise scalar equations used later.

\section{Enhanced logarithmic Hitchin fibers on the line-bundle locus}\label{sec:enhanced-hitchin}

Throughout this section we assume the hypotheses of Theorem~\ref{thm:B}.  Thus \(\overline{Y}\) is an integral compactified spectral curve, the marked spectral scheme
\[
Z\,:=\,\varphi^{-1}(D_{\mathrm{tot}})
\]
is reduced, and the slope maps \(\rho^a\) are defined for all \(a\in D_{\mathrm{tot}}\).  In this situation, Proposition~\ref{prop:marked-graph} identifies the marked condition with the equality
\begin{equation}\label{eq:section5-marked-condition}
\vartheta_Z=m_{b_Z}.
\end{equation}

We now restrict to the line-bundle locus.  The point of this section is that, once \((\overline{Y},Z,b_Z)\) is fixed, the admissible enhancement data are controlled by a single affine scheme independent of the chosen line bundle.

Set
\[
M\,:=\,\varphi^*L(P),
\qquad
M_Z\,:=\,\varphi^*L|_Z,
\qquad
V\,:=\,H^0(\overline{Y},M),
\qquad
V_Z\,:=\,H^0(Z,M_Z).
\]
Restriction to \(Z\), followed by the residual map \(\beta_Z^L\), gives a linear map
\begin{equation}\label{eq:evres-linear-map}
\evres_Z:V\longrightarrow V_Z,
\qquad
s\longmapsto \beta_Z^L(s|_Z).
\end{equation}
Recall from Definition~\ref{def:lambdaZ} that
\[
b_Z\in V_Z
\]
is the distinguished marked section.

\subsection{The fixed affine enhancement scheme}

Let \(\mathcal N\) be a line bundle on \(\overline{Y}\).  An \emph{enhancement} of \(\mathcal N\) is an \(\OO_{\overline{Y}}\)-linear morphism
\[
\vartheta:\mathcal N\longrightarrow \mathcal N\otimes M.
\]
Since \(\mathcal N\) is invertible, there is a canonical identification
\begin{equation}\label{eq:Hom-sections-linebundle-new}
\Hom_{\OO_{\overline{Y}}}\!\bigl(\mathcal N,\,\mathcal N\otimes M\bigr)
\,\cong\,
H^0\!\bigl(\overline{Y},M\bigr)=V.
\end{equation}
Thus every enhancement of \(\mathcal N\) is given by multiplication by a unique section \(s\in V\).

If \(W\) is a finite-dimensional \(\CC\)-vector space, write
\[
\underline{W}\,:=\,\Spec\bigl(\Sym(W^\vee)\bigr)
\]
for the associated affine space.  The map \eqref{eq:evres-linear-map} induces a morphism of affine schemes
\[
\underline{\evres}_Z:\underline{V}\longrightarrow \underline{V_Z}.
\]

\begin{definition}\label{def:fixed-affine-scheme}
The \emph{affine enhancement scheme} is
\begin{equation}\label{eq:A-Delta-scheme}
A_Z
\,:=\,
\Spec\CC\times_{\underline{V_Z}}\underline{V},
\end{equation}
where \(\Spec\CC\longrightarrow \underline{V_Z}\) is the \(\CC\)-point corresponding to \(b_Z\).
\end{definition}

Equivalently, \(A_Z\) is the scheme-theoretic fiber of \(\underline{\evres}_Z\) over \(b_Z\).  In particular, \(A_Z\) depends only on \((\overline{Y},Z,b_Z)\).

\begin{lemma}\label{lem:representability-A-Delta}
For every \(\CC\)-scheme \(T\), there is a natural identification
\[
A_Z(T)
\,\cong\,
\Bigl\{
s\in V\otimes_\CC \Gamma(T,\,\OO_T)\;\Bigm|\;
(\evres_Z\otimes 1)(s)=b_Z\otimes 1
\Bigr\}.
\]
\end{lemma}
\begin{proof}
For any finite-dimensional \(\CC\)-vector space \(W\), morphisms \(T\longrightarrow \underline{W}\) are the same as global sections of \(W\otimes_\CC \OO_T\), i.e.
\[
\Hom(T,\,\underline{W})\,\cong\, W\otimes_\CC \Gamma(T,\,\OO_T).
\]
Applying this to \(\underline{V}\) and \(\underline{V_Z}\), and then using the fiber-product description \eqref{eq:A-Delta-scheme}, gives the result.
\end{proof}

\begin{proposition}\label{prop:fixed-linebundle-affine-new}
Let \(\mathcal N\) be a line bundle on \(\overline{Y}\).  Under the identification \eqref{eq:Hom-sections-linebundle-new}, the admissible enhancements
\[
\vartheta:\mathcal N\longrightarrow \mathcal N\otimes M
\]
are naturally identified with \(A_Z(\CC)\).
\end{proposition}
\begin{proof}
Let \(\vartheta\) correspond to \(s_\vartheta\in V\).  Since \(\mathcal N|_Z\) is invertible, there is a canonical identification
\[
\sheafHom_{\OO_Z}\!\bigl(\mathcal N|_Z,\,\mathcal N|_Z\otimes M_Z\bigr)\,\cong\, M_Z.
\]
Under this identification, the morphism
\[
\vartheta_Z=(\id_{\mathcal N|_Z}\otimes \beta_Z^L)\circ \vartheta|_Z
\]
corresponds exactly to \(\evres_Z(s_\vartheta)\in V_Z\).  Therefore
\[
\vartheta_Z=m_{b_Z}
\iff
\evres_Z(s_\vartheta)=b_Z.
\]
By \eqref{eq:section5-marked-condition}, the left-hand side is precisely the admissibility condition.  Hence admissible enhancements are exactly the \(\CC\)-points of \(A_Z\).
\end{proof}

The next lemma identifies the translation space of \(A_Z\).

\begin{lemma}\label{lem:evres-kernel-new}
Let \(i_Z:Z\hookrightarrow \overline{Y}\) be the closed immersion, and set
\[
Z_D\,:=\,\varphi^{-1}(D),
\qquad
Z_P\,:=\,\varphi^{-1}(P).
\]
Then there is a natural short exact sequence of coherent sheaves on \(\overline{Y}\)
\begin{equation}\label{eq:evres-sheaf-sequence-newer}
0\longrightarrow \varphi^*L(-Z_D)
\longrightarrow M
\xrightarrow{\ \eta_Z\ }
{i_Z}_*(M_Z)
\longrightarrow 0,
\end{equation}
whose induced map on global sections is \(\evres_Z\).  Consequently,
\[
\ker(\evres_Z)\,\cong\, H^0\bigl(\overline{Y},\,\varphi^*L(-Z_D)\bigr).
\]
In particular, if \(A_Z\neq\varnothing\), then \(A_Z\) is a torsor under the additive group attached to \(H^0(\overline{Y},\,\varphi^*L(-Z_D))\).
\end{lemma}
\begin{proof}
Because \(Z\) is reduced and finite over the curve \(X\), the schemes \(Z_D\) and \(Z_P\) are reduced effective divisors on the integral curve \(\overline{Y}\), they are disjoint, and
\[
Z=Z_D\sqcup Z_P.
\]
Since \(P\) is reduced, one has
\[
\varphi^*P=Z_P,
\qquad\text{hence}\qquad
M=\varphi^*L(Z_P).
\]

For any reduced effective divisor \(E\) on a curve and any line bundle \(N\), there is an exact sequence
\[
0\longrightarrow N\longrightarrow N(E)\longrightarrow {i_E}_*(N|_E)\longrightarrow 0.
\]
Applying this first with \(N=\varphi^*L\) and \(E=Z_P\), and then with \(N=\varphi^*L\) and \(E=Z_D\), one obtains
\[
0\longrightarrow \varphi^*L
\longrightarrow M
\longrightarrow {i_{Z_P}}_*(\varphi^*L|_{Z_P})
\longrightarrow 0
\]
and
\[
0\longrightarrow \varphi^*L(-Z_D)
\longrightarrow \varphi^*L
\longrightarrow {i_{Z_D}}_*(\varphi^*L|_{Z_D})
\longrightarrow 0.
\]
Since \(Z_D\cap Z_P=\varnothing\), these quotient maps combine to a surjection
\[
\eta_Z:M\longrightarrow {i_Z}_*(M_Z)
\,\cong\,
{i_{Z_D}}_*(\varphi^*L|_{Z_D})\oplus {i_{Z_P}}_*(\varphi^*L|_{Z_P}).
\]
Its kernel consists exactly of those sections of \(M\) with zero principal part along \(Z_P\) and vanishing restriction along \(Z_D\), namely \(\varphi^*L(-Z_D)\).  This proves \eqref{eq:evres-sheaf-sequence-newer}.  On global sections, \(\eta_Z\) is precisely the map \(\evres_Z\).

The final statement is immediate: every nonempty fiber of a linear map of affine spaces is a torsor under its kernel.
\end{proof}

\subsection{Families and functoriality}

Let
\[
p_Y:\overline{Y}\times T\longrightarrow \overline{Y},
\qquad
p_T:\overline{Y}\times T\longrightarrow T
\]
be the projections.

Define \(\mathfrak H^{\log,\,\mathrm{lb}}_{\overline{Y},d}\) to be the category fibered in groupoids over \(\CC\)-schemes whose objects over a scheme \(T\) are pairs \((\mathcal N,\,\vartheta)\) such that:
\begin{enumerate}
\item[(i)] \(\mathcal N\) is a line bundle on \(\overline{Y}\times T\) whose restriction to every geometric fiber of \(p_T\) has degree \(d\);
\item[(ii)] \(\vartheta:\mathcal N\longrightarrow \mathcal N\otimes p_Y^*M\) is an \(\OO_{\overline{Y}\times T}\)-linear morphism whose restriction to every geometric fiber is admissible.
\end{enumerate}
Morphisms are isomorphisms of line bundles intertwining the enhancement morphisms.

Let \(\mathbf{Pic}^d(\overline{Y})\) denote the Picard stack of degree-\(d\) line bundles on \(\overline{Y}\).

\begin{proposition}\label{prop:family-enhancement}
Let \(T\) be a \(\CC\)-scheme, and let \(\mathcal N\) be a line bundle on \(\overline{Y}\times T\).  Then:
\begin{enumerate}
\item[(a)] relative enhancements
\(
\vartheta:\mathcal N\longrightarrow \mathcal N\otimes p_Y^*M
\)
are canonically equivalent to morphisms
\(
s_\vartheta:T\longrightarrow \underline{V};
\)
\item[(b)] under this correspondence, \(\vartheta\) is admissible on every geometric fiber if and only if \(s_\vartheta\) factors through the closed subscheme
\(
A_Z\hookrightarrow \underline{V}.
\)
Both constructions are functorial in \(T\).
\end{enumerate}
\end{proposition}
\begin{proof}
Since \(\mathcal N\) is invertible,
\[
\sheafHom\!\bigl(\mathcal N,\,\mathcal N\otimes p_Y^*M\bigr)\,\cong\, p_Y^*M.
\]
Because \(\overline{Y}\) is proper over \(\Spec\CC\), flat base change gives
\[
p_{T*}(p_Y^*M)\,\cong\, H^0(\overline{Y},M)\otimes_\CC \OO_T=V\otimes_\CC \OO_T.
\]
Therefore
\[
\Hom_{\OO_{\overline{Y}\times T}}\!\bigl(\mathcal N,\,\mathcal N\otimes p_Y^*M\bigr)
\,\cong\,
\Gamma(T,V\otimes_\CC \OO_T)
\,\cong\,
\Hom(T,\,\underline{V}),
\]
which proves \textup{(a)}.

Now let \(s_\vartheta:T\longrightarrow \underline{V}\) be the morphism corresponding to \(\vartheta\).  By Lemma~\ref{lem:representability-A-Delta}, the factorization of \(s_\vartheta\) through \(A_Z\) is equivalent to the equality
\[
(\evres_Z\otimes 1)(s_\vartheta)=b_Z\otimes 1
\]
in \(V_Z\otimes_\CC \Gamma(T,\,\OO_T)\).  Pulling back to a geometric point \(t\longrightarrow T\), this becomes precisely the condition that the fiber enhancement \(\vartheta_t\) satisfy
\[
(\vartheta_t)_Z=m_{b_Z}.
\]
By \eqref{eq:section5-marked-condition}, that is the admissibility condition on the fiber over \(t\).  Hence \textup{(b)} follows.
\end{proof}

The key point is that the admissibility condition is cut out by the fixed closed subscheme \(A_Z\subset \underline{V}\), independently of the chosen line bundle.

\begin{theorem}\label{thm:stack-product}
Viewing \(A_Z\) as a category fibered in groupoids through its functor of points, there is a canonical equivalence of categories fibered in groupoids
\begin{equation}\label{eq:stack-product}
\mathfrak H^{\log,\,\mathrm{lb}}_{\overline{Y},d}
\;\simeq\;
\mathbf{Pic}^d(\overline{Y})\times A_Z.
\end{equation}
Under this equivalence, the forgetful morphism
\[
\mathfrak H^{\log,\,\mathrm{lb}}_{\overline{Y},d}\longrightarrow \mathbf{Pic}^d(\overline{Y}),
\qquad
(\mathcal N,\,\vartheta)\longmapsto \mathcal N,
\]
identifies with the projection to the first factor.
\end{theorem}
\begin{proof}
For a \(T\)-point \((\mathcal N,\,\vartheta)\) of \(\mathfrak H^{\log,\,\mathrm{lb}}_{\overline{Y},d}\), let
\[
s_\vartheta:T\longrightarrow \underline{V}
\]
be the morphism corresponding to \(\vartheta\) under Proposition~\ref{prop:family-enhancement}\textup{(a)}.  By Proposition~\ref{prop:family-enhancement}\textup{(b)}, the fiberwise admissibility of \(\vartheta\) is equivalent to the factorization of \(s_\vartheta\) through \(A_Z\).  Hence we obtain a functor
\[
\mathfrak F:
\mathfrak H^{\log,\,\mathrm{lb}}_{\overline{Y},d}
\longrightarrow
\mathbf{Pic}^d(\overline{Y})\times A_Z,
\qquad
(\mathcal N,\,\vartheta)\longmapsto (\mathcal N,s_\vartheta).
\]

Conversely, given a \(T\)-point \((\mathcal N,s)\) of \(\mathbf{Pic}^d(\overline{Y})\times A_Z\), compose \(s\) with the closed immersion \(A_Z\hookrightarrow \underline{V}\).  Proposition~\ref{prop:family-enhancement}\textup{(a)} then gives a unique relative enhancement
\[
\vartheta_s:\mathcal N\longrightarrow \mathcal N\otimes p_Y^*M,
\]
and Proposition~\ref{prop:family-enhancement}\textup{(b)} shows that \(\vartheta_s\) is admissible on every geometric fiber.  Thus \((\mathcal N,\,\vartheta_s)\) defines an object of \(\mathfrak H^{\log,\,\mathrm{lb}}_{\overline{Y},d}(T)\).  This gives a functor
\[
\mathfrak G:
\mathbf{Pic}^d(\overline{Y})\times A_Z
\longrightarrow
\mathfrak H^{\log,\,\mathrm{lb}}_{\overline{Y},d}.
\]

The constructions \(\mathfrak F\) and \(\mathfrak G\) are functorial in \(T\), and they are inverse to one another by construction.  This proves \eqref{eq:stack-product}.  The final compatibility with the forgetful morphism is immediate.
\end{proof}

\begin{corollary}\label{cor:algebraicity-enhanced-stack}
The category fibered in groupoids \(\mathfrak H^{\log,\,\mathrm{lb}}_{\overline{Y},d}\) is an algebraic stack.  Moreover, the equivalence \eqref{eq:stack-product} upgrades canonically to an equivalence of algebraic stacks
\[
\mathfrak H^{\log,\,\mathrm{lb}}_{\overline{Y},d}
\;\simeq\;
\mathbf{Pic}^d(\overline{Y})\times A_Z.
\]
\end{corollary}
\begin{proof}
The Picard stack \(\mathbf{Pic}^d(\overline{Y})\) is algebraic, and \(A_Z\) is an affine \(\CC\)-scheme.  Hence \(\mathbf{Pic}^d(\overline{Y})\times A_Z\) is an algebraic stack.  The claim follows from Theorem~\ref{thm:stack-product}.
\end{proof}

\end{document}